\documentclass[12pt,a4paper,leqno,fleqn]{amsart}
\usepackage[cp1250]{inputenc}

\usepackage{amsmath}
\usepackage{amsthm}
\usepackage{amssymb}

\usepackage[all]{xy}
\DeclareMathAlphabet{\mathpzc}{OT1}{pzc}{m}{it}

\newcommand{\id}{\operatorname{id}}

\newcommand{\calB}{\mathcal{B}}
\newcommand{\calV}{\mathcal{V}}
\newcommand{\calP}{\mathcal{P}}
\newcommand{\calR}{\mathcal{R}}
\newcommand{\calU}{{\mathcal{U}}}

\newcommand{\cl}{\operatorname{cl}}

\newcommand{\Lim}{\mathit{Lim}}

\newcommand{\calQ}{\mathcal{Q}}

\renewcommand{\epsilon}{\varepsilon}
\renewcommand{\int}{\operatorname{int}}
\renewcommand{\phi}{\varphi}

\newcounter{diagram}
\newtheorem{thm}{Theorem}
\newtheorem{pro}[thm]{Proposition}
\newtheorem{lem}[thm]{Lemma}
\newtheorem*{lem*}{Lemma}
\newtheorem*{cor*}{Corollary}
\newtheorem{cor}[thm]{Corollary}

\makeatletter
\@namedef{subjclassname@2020}{
  \textup{2020} Mathematics Subject Classification}
\makeatother
\title{Dimensional types and  $P$-spaces}


\author{Wojciech Bielas}
\address{Institute of Mathematics, University of Silesia in Katowice, ul. Bankowa 14, 40-007 Katowice}
\email{wojciech.bielas@us.edu.pl}
\email{andrzej.kucharski@us.edu.pl}
\email{szymon.plewik@us.edu.pl}
\author{Andrzej Kucharski}
\author{Szymon Plewik}
\date{\today}
\begin{document}

\maketitle

\begin{abstract}
We investigate the category of discrete topological spaces, with emphasis on inverse systems of height $\omega_1$.
Their inverse limits belong to the class of $P$-spaces, which allows us to explore dimensional types of these spaces.
\end{abstract}

\section{Introduction}
The purpose of this paper is to discuss the connection between inverse limits of height $\omega_1$, which extend the category of discrete topological spaces.
Results are concentrated on  dimensional types of some $P$-spaces.
The name ``$P$-space'' was used  by L. Gillman and M. Henriksen \cite{gil}.
If a space $X$ is completely regular and every countable intersection of open sets of $X$ is open, then $X$ is called  $P$\textit{-space}.
  A. K. Misra \cite{mis} proposed investigation of $T_1$-spaces which satisfy this last condition and called them $P$\emph{-spaces}, too.
  If $X$ is a  $T_1$-space, then $X$  endowed with the topology generated by all $G_\delta$-sets  is called the $G_\delta$\emph{-modification} of $X$, and it is denoted  by $(X)_\delta$.
  Thus, any $T_1$-space is a $P$-space if and only if it is  its own $G_\delta$-modification.  
Following  M. Fr\'echet \cite{fre}, K. Kuratowski \cite{kur} or W. Sierpi\'nski \cite{sie}, etc., we are convinced that pairs of topological spaces which embed into each other are interesting in themselves.
This relation appeared under various names:  \textit{topological rank}, see \cite[p. 112]{kur}; \textit{dimensional type}, see \cite[p. 130]{sie}, etc.
If $X$ is homeomorphic to a subspace of $Y$, i.e. $Y$ contains a homeomorphic copy of $X$, in short $X\subset_hY$, then  the dimensional type of $X$ is less or equal to the dimensional type of $Y$.
If  $X\subset_hY$ and $Y\subset_hX$, then we  write $X=_hY$.
If $X\subset_hY$ and $Y$ does not contain a homeomorphic copy of $X$, then $X$ has smaller dimensional type than $Y$. 
In the paper \cite{com}, the relation ``$\subset_h$'' was called ``topological inclusion relation'' and was used to examine topological arrow relation of the form $X\to (Y)^1_2$.
One of the results of this paper, which improves some results from \cite{dow}, is $(2^{\omega_1})_\delta\to (\mathbf{\Sigma })^1_2$, see Section \ref{sec:7}.

We use the standard notation and terminology of \cite{kur}, \cite{jec}, or \cite{eng}, as well as of papers \cite{kun} and \cite{dow} with some minor changes.
A dense in itself space is called  a \textit{crowded} space.
A \textit{partition} is a cover consisting of pairwise disjoint open sets, hence elements of a partition are clopen sets, here a \textit{clopen} set means a closed and open set.
We refer the readers to the book \cite[pp. 98--104]{eng} for details about  limits of  inverse systems.
We use inverse systems $\{X_\alpha,\pi^\alpha_\beta,\omega_1\}$, where each  $X_\alpha$ is a  discrete space.

The paper is organised as follows.
In Section \ref{sec:2} we  briefly sketch some  facts about $P$-spaces.
For the description of inverse systems, see   \cite{eng}, but notions related to trees are taken from \cite{jec}.
In Sections \ref{sec:3} and \ref{sec:4}, we discuss the $G_\delta$-modifications of Cantor cubes and its subspaces, where we show that the Baire number varies among them.
Theorem \ref{thm:14} says that, under the Continuum Hypothesis, any two dense subspaces of $(2^{\omega_1})_\delta$ of cardinality $\frak{c}$ are homeomorphic.
In Sections \ref{sec:5}--\ref{sec:7}, we investigate  $P$-spaces of cardinality $\omega_1$, $P$-spaces of weight $\omega_1$ and $P$-spaces of cardinality and weight $\omega_1$.
Any $P$-space of weight $\omega_1$ can be embedded into $(2^{\omega_1})_\delta$, hence the space $(2^{\omega_1})_\delta$ has the greatest dimensional type in the class of all $P$-spaces of weight $\omega_1$, whenever the Continuum Hypothesis
is assumed.
In Section \ref{sec:8}, we introduce the notion of a $\lambda$-thin labeling and prove that any  two $P$-spaces which have a $\lambda$-thin labeling are homeomorphic.
Also we prove that if a $P$-space $X$ has an $\omega$-thin labeling (or an $\omega_1$-thin labeling), then $X$ has the smallest dimensional type in the class of all crowded $P$-spaces of weight $\omega_1$.
Finally, we give a few remarks about rigid Lindel\"of $P$-spaces.

\section{$P$-spaces, inverse limits and trees}\label{sec:2}

Before proceeding, let us note the following observation.

\begin{pro}\label{pro:3}
Any regular $P$-space is  $0$-dimensional.
\end{pro}
\begin{proof}
Fix a regular $P$-space $X$ and  $x\in V\subseteq X$, where $V$ is an open set.
By regularity of $X$, there exists a sequence $(U_n)$ of open sets such that $x\in \cl U_{n}\subseteq U_{n-1}.$
  The set
$$V^* = \bigcap_n \cl U_n = \bigcap_n U_n \subseteq V$$
is clopen and $x\in V^*\subseteq V$.
Therefore, $X$ has a base consisting of clopen sets.
\end{proof}

The family of all  clopen sets of a regular $P$-space is a $\sigma$-algebra.
For these reasons, from now on  we assume that a $P$-space is completely regular.

\begin{pro}
If $X$ is a  $P$-space and  a clopen subset $U\subseteq X$ has a limit point, then $U$ contains uncountably many pairwise disjoint clopen subsets.
\end{pro}

\begin{proof}
  Let $U\subseteq X$ be a clopen set with a limit point $y\in U$.
By Proposition \ref{pro:3}, there exists an uncountable base $\calB$ at $y$ consisting of clopen subsets.
Choose a strictly decreasing sequence
$$\{V_\alpha\subseteq U\colon\alpha<\omega_1\}\subseteq\calB.$$
The family $\{V_\alpha\setminus V_{\alpha+1}\colon \alpha<\omega_1\}$ is as desired.
\end{proof}


\begin{pro}\label{cor:5}
In a  $P$-space, any countable family consisting of open covers has a common refinement.
\end{pro}

\begin{proof}
It suffices to consider a family  $\{\calP_n\colon n<\omega\}$ of covers, each one consists of  clopen sets.
  For a point $x$, choose  $V_{n,x}\in\calP_n$ such that $x\in V_{n,x}$.
  The intersection
  $$V_x=\bigcap\{V_{n,x}\colon n<\omega\}$$
  is a clopen set, so the family of  all $V_x$ is a desired refinement.
\end{proof}

  Let $\{X_\alpha\colon \alpha<\omega_1\}$ be a family of discrete spaces.
The $G_\delta$-modification  $(\prod_{\alpha<\omega_1}X_\alpha)_\delta$ of a product   $\prod_{\alpha<\omega_1}X_\alpha$ with the Tychonoff topology has a base
  $$\{[f]\colon f\in \prod_{\beta<\alpha}X_\beta\mbox{ and } \alpha<\omega_1\},\mbox{ where }[f]=\{g\in\prod_{\alpha<\omega_1}X_\alpha\colon f\subseteq g\}.$$
Assume that there are given bonding maps  $\pi^\alpha_\beta\colon X_\alpha\to X_\beta$ such that
$$\gamma<\beta<\alpha<\omega_1\mbox{ implies }\pi^\beta_\gamma\circ \pi^\alpha_\beta=\pi^\alpha_\gamma.$$
We have an inverse system $\mathbb{P}=\{X_\alpha,\pi^\alpha_\beta,\omega_1\}$ and the inverse limit $\varprojlim\mathbb{P}$.
By \cite[Proposition 2.5.5.]{eng}, the inverse limit $\varprojlim\mathbb{P}$ is a $P$-space.
Given an inverse system $\mathbb{P}=\{X_\alpha,\pi^\alpha_\beta,\omega_1\}$, we would like to enrich it by injections $\iota^\alpha_\beta \colon X_\beta\to X_\alpha$, for each $\beta<\alpha$, such that
\begin{itemize}
\item[(a)] $\iota^\alpha_\beta\circ\iota^\beta_\gamma=\iota^\alpha_\gamma$, where $\gamma<\beta<\alpha<\omega_1$;
\item[(b)]  $\pi^\alpha_\beta\circ \iota^\alpha_\beta=\id_{X_\beta}$, where $\beta<\alpha<\omega_1$,
\end{itemize}
see Diagram \ref{fig:1}.

\begin{figure}[h]
\centering
$$\xymatrix{
	X_\gamma\ar@/^2.5pc/@{--{>}}[rrrr]^{\iota^\alpha_\gamma} \ar@<1ex>@{--{>}}[rr]^{\iota^\beta_\gamma} && X_\beta\ar[ll]^{\pi^\beta_\gamma}\ar@<1ex>@{--{>}}[rr]^{\iota^\alpha_\beta}&& X_\alpha\ar[ll]^{\pi^\alpha_\beta}\ar@/^2pc/[llll]^{\pi^\alpha_\gamma}.
      }$$
      \caption{}      \label{fig:1}
      \end{figure}

      Given commutative diagrams, as Diagram \ref{fig:1} with $\gamma\leq\beta\leq\alpha<\omega_1$,  we get an enriched inverse system $\mathbb{P}=\{X_\alpha,\pi^\alpha_\beta,\iota^\alpha_\beta,\omega_1\}$.
      We use the same symbol as for the system $\{X_\alpha,\pi^\alpha_\beta,\omega_1\}$, since both of them gives the same inverse limit.
      But the injections allow us to define the following  $P$-spaces.
Let $$\mathbf{\Sigma}_\mathbb{P}=\{(p_\alpha)\in\varprojlim\mathbb{P}\colon\exists_{\gamma<\omega_1}\forall_{\beta>\gamma}\;p_\beta=\iota^\beta_\gamma(p_\gamma)\}\subseteq\varprojlim\mathbb{P} \subseteq\prod_{\alpha<\omega_1}X_\alpha$$
and let $\mathbf{\Sigma}^{\mathbb{P}}$ be a subspace of $\mathbf{\Sigma}_{\mathbb{P}}$, which consists of threads $(p_\alpha)\in \mathbf{\Sigma}_{\mathbb{P}}$ such that if $\alpha$ is an infinite limit ordinal, then there exists $\beta<\alpha$ such that $p_\alpha=\iota^\alpha_\beta(p_\beta)$. 
We say that a thread $(p_\alpha)\in\varprojlim\mathbb{P}$ has a \textit{jump} at an ordinal $\beta$, whenever  $\iota^{\beta+1}_\beta(p_\beta)\neq p_{\beta+1}$.

\begin{lem}\label{lem:I5}
Any thread in $\mathbf{\Sigma}^{\mathbb{P}}$ has finitely many jumps.
\end{lem}

\begin{proof}
Suppose that a thread $(p_\alpha)$ belongs to $\mathbf{\Sigma}^{\mathbb{P}}$.
If there exists a limit ordinal $\gamma$ which is a supremum of infinitely many jumps for this thread, then    $\iota^\gamma_\beta(p_\beta)\neq p_\gamma$             for any $\beta<\gamma$; a contradiction with $(p_\alpha)\in\mathbf{\Sigma}^\mathbb{P}$.
\end{proof}

Let $\mathbb{Q}= \{Q_\alpha,\pi^\alpha_\beta,\iota^\alpha_\beta,\omega_1\}$ and $\mathbb{R}=\{R_\alpha,r^\alpha_\beta,e^\alpha_\beta, \omega_1\}$ be enriched inverse sequences.

\begin{figure}[h]
    $$\xymatrix{
	Q_\beta \ar@<1ex>@{-->}[rr]^{\iota^\alpha_\beta} \ar@{-->}[d]_{s_\beta} && Q_\alpha\ar[ll]^{\pi^\alpha_\beta} \ar@{-->}[d]^{s_\alpha} \\
	R_\beta \ar@<1ex>@{-->}[rr]^{e^\alpha_\beta}      && R_\alpha\ar[ll]^{r^\alpha_\beta}
      }$$\caption{}\label{fig:2}
      \end{figure}\bigskip
      
\begin{lem}\label{lem:I}
Let $s_\alpha\colon Q_\alpha \to R_\alpha$ be one-to-one functions   such that Diagram \ref{fig:2} is commutative, whenever $\beta<\alpha<\omega_1$, then
$$\Sigma_{\mathbb{Q}}\subset_h \Sigma_{\mathbb{R}}\mbox{ and }\Sigma^{\mathbb{Q}}\subset_h \Sigma^\mathbb{R},\mbox{  and }\varprojlim\mathbb{Q}\subset_h \varprojlim\mathbb{R}.$$
  Moreover,  if each $s_\alpha$ is a bijection, then 
  $\Sigma_{\mathbb{Q}}$ is homeomorphic to $\Sigma_{\mathbb{R}}$ and   $\Sigma^{\mathbb{Q}}$ is homeomorphic to $\Sigma^\mathbb{R}$, and  $\varprojlim\mathbb{Q}$ is homeomorphic to $\varprojlim\mathbb{R}$.
\end{lem}

\begin{proof}
If $(u_\alpha)\in\varprojlim\mathbb{Q}$, then the formula $(u_\alpha)\mapsto (s_\alpha(u_\alpha))$ defines an embedding from $\varprojlim\mathbb{Q}$ to $\varprojlim\mathbb{R}$.
The  restrictions of this embedding give embeddings $\Sigma_{\mathbb{Q}}\to \Sigma_\mathbb{R}$ and $\Sigma^{\mathbb{Q}}\to \Sigma^\mathbb{R}$.
But if all $s_\alpha$ are bijections, then these embeddings are homeomorphisms.
\end{proof}

Any inverse system $\mathbb{P}=\{X_\alpha,\pi^\alpha_\beta,\omega_1\}$ can be interpreted as a tree of height $\omega_1$, we refer the readers for basic notions about trees to the book \cite{jec}.
Namely, assume that the sets $X_\alpha$ are pairwise disjoint.
Let
$$T=\bigcup\{X_\alpha\colon\alpha<\omega_1\}$$
and we put $x\leq y$, whenever $x\in X_\alpha$ and $y\in X_\beta$,  and  $x=\pi^\alpha_\beta(y)$.
Let $[T]$ be the family of all branches of length $\omega_1$.
If $A\in[T]$,  then $A=\{p_\alpha\colon p_\alpha\in X_\alpha\mbox{ and }\alpha<\omega_1\}$ and  $(p_\alpha)\in\varprojlim\mathbb{P}$.
So, the mapping $A\mapsto (p_\alpha)$ is a bijection between $[T]$ and $\varprojlim\mathbb{P}$, which is also a homeomorphism, whenever $[T]$ is endowed with the topology generated by the family
$$\{\{b\in [T]\colon x\in b\}\colon x\in T\}.$$
Some authors use a notion \textit{tree topology} for the topology just defined on $[T]$, compare \cite[p. 14]{rud}.

The interpretation of an inverse limit as  a tree which yields a topological space,  consisting of branches of length $\omega_1$, leads us to the notion of labeling.
Namely, a surjection $E\colon T\to Y\subseteq [T]$ is called \textit{labeling}, if for every $x,y\in T$ we have $x\in E(x)$ and the following implication holds:
$$x\leq y\mbox{ and }y\in E(x) \Rightarrow E(x)=E(y).$$

If $E\colon T\to Y\subseteq [T]$ is a labeling, then $(T,\leq,E)$ is a \textit{labelled tree},  which corresponds to the enriched inverse system $\{X_\alpha,\pi^\alpha_\beta,\iota^\alpha_\beta, \omega_1\}$, where 
$X_\alpha$ is the $\alpha$-th level of $T$, and $\pi^\alpha_\beta\colon X_\alpha\to X_\beta$ is such that $\pi^\alpha_\beta(x)\leq x$ for every $x\in X_\alpha$.
The injections $\iota^\alpha_\beta$ are defined as follows.
If $\alpha<\beta$ and $x\in X_\alpha$, then $\iota^\beta_\alpha(x)$ is the unique element of $E(x)\cap X_\beta$.
Finally $Y=\{E(x)\colon x\in T\}$. 

\section{The $G_\delta$-modifications of the Cantor cubes}\label{sec:3}

For an infinite cardinal $\kappa$, let $2^{\kappa}$ denote the Cantor cube with the product topology.
Thus  $(2^{\kappa})_\delta$ is the $G_\delta$-modification of the Cantor cube $2^{\kappa}$.
Recall that if $f\colon A\to \{0,1\}$ is a function and $A\subseteq \kappa$, then  $[f]$ denotes the family of all extensions of $f$ with  the domain  $\kappa$ and values in $\{0,1\}$.
The next lemma is well-known, for example for  readers  familiar with box products.
\begin{lem}\label{lem:9}
If $\kappa$ is an infinite cardinal, then the family
$$\{[f]\colon f\in 2^A\mbox{ and } A\in [\kappa]^{\omega}\}$$
is a base for  $(2^{\kappa})_\delta$.\qed
\end{lem}

Recall that the \textit{Baire number} of a crowded topological space $X$ is the smallest cardinal $\kappa$ such that $X$ cannot be covered by a family of cardinality less than  $\kappa$ and consisting of nowhere dense subsets.
The Baire number of $(2^{\kappa})_\delta$ is always at least $\omega_2$.

\begin{pro}\label{pro:9a}
If $\kappa$ is an uncountable cardinal, then any union of at most $\omega_1$  nowhere dense subsets of $(2^{\kappa})_\delta$ is a boundary set.
\end{pro}
\begin{proof}
  Let $\mathcal{F}=\{F_\beta\colon \beta<\omega_1\}$ be a family of nowhere dense subsets of $(2^{\kappa})_\delta$.
  Fix a non-empty open set $U\subseteq (2^\kappa)_\delta$ and then choose a function $f_0\colon A_0\to\{0,1\}$ such that $[f_0]\subseteq U\setminus F_0$ and $A_0\in [\kappa]^{\omega}$.
Suppose that for $\beta<\alpha$ there are defined functions $f_\beta\colon A_\beta\to\{0,1\}$   such that  $[f_\gamma]\supseteq [f_\beta]$ and $A_\beta\in[\kappa]^{\omega}$ whenever $\gamma<\beta<\alpha$.
Choose $A_\alpha\in[\kappa]^{\omega}$ and  $f_\alpha\colon A_\alpha\to\{0,1\}$ such that
$$[f_\alpha]\cap F_\alpha=\emptyset\mbox{ and }\bigcup\{f_\beta\colon \beta<\alpha\}\subseteq f_\alpha.$$
  If
  $\bigcup\{f_\beta\colon \beta<\omega_1\}\subseteq f\in 2^\kappa$, then  $f\in U\setminus\bigcup\mathcal{F}$, i.e. the complement of $\bigcup\mathcal{F}$ is a  dense subset of $(2^\kappa)_\delta$.
\end{proof}

\section{The $G_\delta$-modification of  $2^{\omega_1}$ }\label{sec:4}
For each $\alpha<\omega_1$, assume that the set $2^{\alpha}$ is equipped with the discrete topology  and let $\pi^\alpha_\beta(f)=f|_\beta$, whenever $\beta<\alpha$ and $f\in 2^\alpha$.
We have 
$$\prod_{\alpha<\omega_1}2^\alpha\supseteq\varprojlim\{2^\alpha, \pi^\alpha_\beta, \omega_1\}\stackrel{\phi}{\to} (2^{\omega_1})_\delta,$$
where $\phi((p_\alpha))=\bigcup\{p_\alpha\colon\alpha<\omega_1\}$ for  each $(p_\alpha)\in \varprojlim\{2^\alpha,\pi^\alpha_\beta,\omega_1\}$.
Since $\phi(\pi^{-1}_\alpha(p_\alpha))=[p_\alpha]$, then the bijection $\phi$ is a homeomorphism. 

If  $\alpha<\omega_1$ and $f\in 2^\alpha$, then  put
$$f^*(\gamma)=
\begin{cases}
f(\gamma),&\mbox{when }\gamma<\alpha;\\
0,&\mbox{when }\alpha \leq \gamma< \omega_1.
\end{cases}$$
Let $\mathbf{\Sigma}=\{f^*\colon f\in 2^\alpha\mbox{ and }\alpha<\omega_1\}$.
Thus, if any $X_\alpha=2^\alpha$ is endowed with the discrete topology and $\iota^\alpha_\beta(f)=f^*|_\alpha$, then $\mathbf{\Sigma}=\mathbf{\Sigma}_{\mathbb{P}}$, where  $\mathbb{P}=\{2^\alpha,\pi^\alpha_\beta,\iota^\alpha_\beta,\omega_1\}$.
Clearly, the mapping $[f]\mapsto f^*$ is a labeling.

\begin{pro}\label{pro:9b}
The Baire number of the subspace $\mathbf{\Sigma}\subseteq(2^{\omega_1})_\delta$ is $\omega_1$.
\end{pro}

\begin{proof}
Each set $A_\alpha=\{f^*\colon f\in 2^\alpha\}\subseteq\mathbf{\Sigma}$ is a discrete subset  and $\mathbf{\Sigma}=\bigcup\{A_\alpha\colon\alpha<\omega_1\}$, hence the Baire number of $\mathbf{\Sigma}$ is at most $\omega_1$.

If $\{F_n\colon n<\omega\}$ is an increasing sequence of nowhere dense subsets of $\mathbf{\Sigma}$, then inductively define a function $f_n\colon \alpha_n\to\{0,1\}$  such that $(\alpha_n)$ is an increasing sequence of countable ordinals such that $[f_{n}]\subseteq [f_{n-1}]\setminus F_n$.
We get $f^*\in \mathbf{\Sigma}\setminus\bigcup\{F_n\colon n<\omega\}$, whenever  $f=\bigcup\{f_n\colon n<\omega\}$.
\end{proof}

A small modification of the proof above gives that any first category subset of a dense subspace of $(2^{\omega_1})_\delta$ is nowhere dense.

\begin{cor}
 $(2^{\omega_1})_\delta$ is not homeomorphic to a subspace of $\mathbf{\Sigma}$.
 \end{cor}
 \vspace{-.7cm}
\begin{proof}
Any crowded subspace of $\mathbf{\Sigma}$, being a union of at most $\omega_1$ many discrete subspaces, has the Baire number not greater than $\omega_1$.
But $(2^{\omega_1})_\delta$ has the Baire number at least $\omega_2$.
\end{proof}
Recall that the family $\frak{B}=\{[f]\colon f\in 2^\alpha\mbox{ and }\alpha<\omega_1\}$ is a base for $(2^{\omega_1})_\delta$.
Below, we present a modified proof from \cite{kul}, cf. \cite[3.1. Theorem]{bps}.

\begin{pro}
If $\omega_2\leq\frak{c}$, then the space $(2^{\omega_1})_\delta$ is the union of an increasing sequence $\{D_\alpha\colon \alpha<\omega_2\}$ of nowhere dense subsets. 
\end{pro}

\begin{proof}
If $V=[f]$ and $f\in 2^\alpha$, then let $\{V_\nu\subseteq V\colon \nu<\omega_2\}$ be a family consisting of pairwise disjoint  elements of $\frak{B}$.
Put
$$D_\alpha=(2^{\omega_1})_\delta\setminus\bigcup\{V_\nu\colon \alpha<\nu<\omega_2\mbox{ and }V\in\frak{B}\}.$$
The sequence $\{D_\alpha\colon\alpha<\omega_2\}$ is increasing and its elements are nowhere dense sets.
If $x\in (2^{\omega_1})_\delta$, then $x$ belongs   to elements of the form $[x|_\alpha]\in\frak{B}$ only.
In other words, $x$ belongs to $\omega_1$ many elements  $V\in\frak{B}$.
Thus, there exists $\alpha_x<\omega_2$ such that $x\in D_{\alpha_x}$, which implies $\bigcup\{D_\alpha\colon\alpha<\omega_2\}=(2^{\omega_1})_\delta$.
\end{proof}

In fact, we have shown that the Baire number of $(2^{\omega_1})_\delta$ equals $\omega_2$ is  consistent with ZFC, for example, when  $|2^{\omega_1}|=\omega_2$.

The space $\mathbf{\Sigma}$ is a counterpart of a $P$-space  of cardinality and weight $\omega_1$, which appears in Lemma 2.2 and Corollary 2.3 in \cite{dow}.
If the Continuum Hypothesis fails, the space $\mathbf{\Sigma}$ being of cardinality $\frak{c}$, is not homeomorphic to a $P$-space  of cardinality  $\omega_1$.

\begin{thm}\label{thm:4}
Any dense subset of $(2^{\omega_1})_\delta$   contains a homeomorphic copy of  the space $\mathbf{\Sigma}$.
\end{thm}

\begin{proof}
Let $Y\subseteq(2^{\omega_1})_\delta$ be a dense subset.
Inductively, define a sequence of functions  $S_\alpha\colon 2^\alpha\to 2^\alpha$, for $0<\alpha<\omega_1$, such that the following conditions are fulfilled.
\begin{itemize}
\item[(A).] If $f\in 2^\alpha$, then $S_{\alpha+1}$ restricted to the set $\{f^\frown 0,f^\frown1\}$ is a bijection onto the set $\{S_\alpha(f)^\frown 0,S_\alpha(f)^\frown1\}$, where $f^\frown i=f\cup\{(\alpha,i)\}$.
\item[(B).] If $\alpha$ is a limit ordinal and $f\in 2^\alpha$, then
$$S_\alpha(f)=\bigcup\{S_\beta(f|_\beta)\colon \beta<\alpha\},$$
in particular $S_\alpha(f)\in 2^\alpha$.
\item[(C).] If $g\in\mathbf{\Sigma}$, then
$$\bigcup\{S_\alpha(g|_\alpha)\colon \alpha<\omega_1\}\in Y.$$
\end{itemize}

If $f\in 2^1$, then choose $y(f)\in Y\cap [f]$, and put $S_\alpha(f^*|_\alpha)=y(f)|_\alpha$ for each $\alpha<\omega_1$.
Thus,  $y(f)=\bigcup\{S_\alpha(f^*|_\alpha)\colon \alpha<\omega_1\}$ and $S_1(f)=f$.

Fix $\alpha<\omega_1$ and assume that bijections $S_\beta\colon 2^\beta\to 2^\beta$ are defined,  whenever $\beta<\alpha$, such that conditions (A)--(C) are fulfilled, in particular, for $g\in 2^\beta$ and $\beta<\alpha$, the values $S_\gamma(g^*|_\gamma)$  are defined such that
$$\bigcup\{S_\gamma(g^*|_\gamma)\colon \gamma<\omega_1\}\in Y.$$

If $\alpha$ is a limit ordinal and $g\in 2^\alpha$, and $S_\alpha(g)$ has not  been defined, i.e. $\beta<\alpha$ implies  $g^*\neq (g|_\beta)^*$, then choose
$$y(g)\in Y\cap [\bigcup\{S_\beta(g|_\beta)\colon\beta<\alpha\}]$$
and put $S_\gamma(g^*|_\gamma)=y(g)|_\gamma$ for $\alpha\leq\gamma<\omega_1$.
We get 
$$S_\alpha(g^*|_\alpha)=S_\alpha(g)=\bigcup\{S_\beta(g|_\beta)\colon \beta<\alpha\}\in 2^\alpha$$
and $\bigcup\{S_\gamma(g^*|_\gamma)\colon\gamma<\omega_1\}=y(g)\in Y.$

If $\beta<\alpha$ and $g\in 2^\beta$, then all the values $S_\gamma(g^*|_\gamma)$ are defined by induction assumption.
Since $g^*=(g^\frown0)^*$, it remains to define  $S_\alpha(g^\frown 1)$, where $\alpha=\beta+1$, and $S_\gamma((g^\frown 1)^*|_\gamma)$ for $\alpha<\gamma<\omega_1$.
Namely, let $i\in\{0,1\}$ be such that $S_\beta(g)^\frown i\neq S_{\beta+1}(g^\frown 0)\in 2^{\beta+1}$.
Then  put $$S_{\beta+1}(g^\frown 1)=S_\beta(g)^\frown i,$$
and  choose $y(g)\in Y\cap [S_{\beta+1}(g^\frown 1)]$, and put
$$S_\gamma((g^\frown 1)^*|_\gamma)=y(g)|_\gamma,$$
whenever $\beta+1<\gamma<\omega_1$.
\begin{figure}[h]
$$\xymatrix{
	2^\beta \ar@{->}@{-->}[d]_{S_\beta} && 2^\alpha\ar[ll]_{\pi^\alpha_\beta} \ar@{-->}[d]^{S_\alpha} \\
	2^\beta       && 2^\alpha\ar[ll]_{\pi^\alpha_\beta}
      }$$\caption{}\label{fig:3a}
      \end{figure}

By the definition, Diagram \ref{fig:3a},
where $\pi^\alpha_\beta(f)=f|_\beta$, is commutative.

Equipping each $2^\alpha$ with the discrete topology, we obtain that $(2^{\omega_1})_\delta=\varprojlim\{2^\alpha,\pi^\alpha_\beta,\omega_1\}$ and we get an automorphism
$$S\colon (2^{\omega_1})_\delta\to (2^{\omega_1})_\delta,$$ where $S(f)=\bigcup\{S_\alpha(f|_\alpha)\colon \alpha<\omega_1\}$. 
By condition (C), the image  $S[\mathbf{\Sigma}]\subseteq Y$ is a homeomorphic copy of $\mathbf{\Sigma}$.
\end{proof}

For any partition $(2^{\omega_1})_\delta=A\cup B$, we have $\mathbf{\Sigma} \subset_h A$ or $\mathbf{\Sigma}\subset_h B$, in other words, the topological arrow relation $(2^{\omega_1})\to (\mathbf{\Sigma})^1_2$ is fulfilled.
Namely, there exists $f\in 2^\alpha$,  where $\alpha<\omega_1$,  such that $A\cap [f]$ or $B\cap [f]$ is dense in $[f]$, since both sets $A$ and $B$ cannot be nowhere dense.
But the subspace $[f]\subseteq (2^{\omega_1})_\delta$ is homeomorphic to $(2^{\omega_1})_\delta$.

\begin{cor}\label{cor:7}
The family
$$\{Y\subseteq (2^{\omega_1})_\delta \colon Y\mbox{ is a dense subset}\}$$
 contains the least element with respect to the relation $\subset_h$.\qed
 \end{cor}

But under the Continuum Hypothesis, we conclude the following.

 \begin{thm}\label{thm:14}
If the Continuum Hypothesis is assumed, then any two dense subsets of $(2^{\frak{c}})_\delta$ of cardinality $\frak{c}$ are homeomorphic.
 \end{thm}

 \begin{proof}
  Let  $X=\{x_\alpha\colon \alpha<\omega_1=\frak{c}\}\subseteq(2^{\frak{c}})_{\delta}$ be a dense subset and let $\frak{M}_X=\{\mathcal{P}_\alpha\colon \alpha<\frak{c}\}$, where $\calP_\alpha=\{[f]\cap X\colon f\in 2^\alpha\}$.
  If $V\in\mathcal{P}_\alpha$, since $X\subseteq (2^\frak{c})_\delta$ is dense, then there exists $\beta=\inf\{\nu\colon x_\nu\in V\}$.
  Put  $E(V)=x_\beta$, and then put $E(U)=x_\beta$, whenever $E(V)\in U\subseteq V$ and $U\in\bigcup\frak{M}_X$.
  The function $E\colon\bigcup\{\calP_\alpha\colon\alpha<\frak{c}\}\to X$ is a labeling.
Indeed, fix $\alpha<\omega_1$, then the set $\{x_\gamma\colon\gamma<\alpha\}$ is closed, being countable.
But  $\bigcup\frak{M}_X$ is a base for $X$, hence there exists $\beta<\omega_1$ and $V\in\mathcal{P}_\beta$ such that $x_\alpha\in V$ and $V\cap\{x_\gamma\colon\gamma<\alpha\}=\emptyset$.
  Then $E(V)=x_\alpha$, thus $E$ is a surjection.
We have a labeling $E\colon \bigcup\frak{M}_X\to X$ and bijections $S_\alpha$ defined analogously as in the proof of \ref{thm:4} together with Lemma \ref{lem:I} give needed homeomorphism from $\mathbf{\Sigma}$ to $X$, where $[f]\mapsto f^*$ is a labeling with the image $\mathbf{\Sigma}$.
\end{proof}

In order to avoid a modification of our proof of Theorem \ref{thm:4}, we end this section by stating without proof 
that any two dense subsets of $(2^{\omega_1})_\delta$ which have labelings are homeomorphic.

\section{$P$-spaces of cardinality  $\omega_1$}\label{sec:5}

The below lemma is  a counterpart of \cite[Theorem 3]{pw}.

\begin{lem}\label{thm:6}
If $\mathcal{B}$ is a base consisting of clopen sets of a  $P$-space $X$ of cardinality $\omega_1$, then any open cover of $X$ has a refinement, which is a partition and consists of elements of $\mathcal{B}$.
\end{lem}

\begin{proof}
Let $X = \{x_\alpha\colon \alpha<\omega_1\}$ be a   $P$-space.
Fix an open cover $\mathcal{P}$ of $X$.
Since clopen subsets of $X$ constitute a $\sigma$-algebra, define inductively a desired partition  $\{V_\alpha\colon \alpha<\omega_1\}$ as follows.
If  $x_\alpha\in U\in\calP$ and $x_\alpha\notin \bigcup\{V_\beta\colon \beta<\alpha\}$, then choose $V_\alpha\in\calB$, satisfying $x_\alpha\in V_\alpha\subseteq U$ and $V_\alpha\cap \bigcup\{V_\beta\colon\beta<\alpha\}=\emptyset$, otherwise put $V_\alpha=V_0$.
\end{proof}

Since Proposition \ref{cor:5} and Lemma \ref{thm:6}, we clearly have the following.

\begin{cor}\label{cor:13}
If $X$ is a   $P$-space of cardinality  $\omega_1$, then for any countable family $\mathcal{P}$ consisting of open covers there exists a partition of  $X$, which refines any cover from $\calP$.\qed
\end{cor}

Recall that a cover is point-countable, whenever each point belongs to at most countable many elements of this cover.

\begin{pro}\label{pro:8}
Any base for a  $P$-space of cardinality $\omega_1$ contains a point-countable cover.
\end{pro}

\begin{proof}
Let $X = \{x_\alpha\colon \alpha<\omega_1\}$ be a  $P$-space.
Fix a base $\mathcal{B}$ of open subsets.
Choose $U_0\in\mathcal{B}$ such that $x_0\in U_0$, then choose a clopen set  $V_0$ such that $x_0\in V_0\subseteq U_0$.
Suppose that sets $U_\beta\in\mathcal{B}$ and clopen sets $V_\beta$ are defined for $\beta<\alpha$ in such a way that
  \begin{itemize}
  \item $\{x_\beta\colon \beta<\alpha\}\subseteq\bigcup\{V_\beta\colon \beta<\alpha\}$;
  \item If $\gamma<\beta<\alpha$, then $V_\gamma\cap U_\beta=\emptyset$.
\end{itemize}
If  $\delta<\omega_1$ is the minimal ordinal such that $x_{\delta}\notin \bigcup\{V_\beta\colon \beta<\alpha\}$, then choose $U_\alpha\in\mathcal{B}$ such that
$$x_{\delta}\in U_\alpha\mbox{ and } U_\alpha\cap \bigcup\{V_\beta\colon \beta<\alpha\}=\emptyset$$
and fix a clopen set $V_\alpha$ such that $x_{\delta}\in V_\alpha\subseteq U_\alpha$.
By the definition, the family $\{V_\alpha\colon \alpha<\omega_1\}$ is a partition on $X$.
But the family $\{U_\alpha\colon \alpha<\omega_1\}\subseteq\mathcal{B}$ is a point-countable open cover.
Indeed, if $x\in V_\alpha$, then $x\notin U_\gamma$ for each $\gamma>\alpha$, hence the set $\{\beta\colon x\in U_\beta\}$ is countable.
\end{proof}

As far as we are aware, covering properties of $P$-spaces have not been deeply  investigated.
If a $P$-space is of cardinality or weight $\omega_1$, then it is paracompact.
We do not know when a paracompact $P$-space is totally paracompact, for the definition of totally paracompactness, see \cite{lel}.

\section{$P$-spaces of weight  $\omega_1$}\label{sec:6}
If $X$ is a $P$-space of weight $\omega_1$, then any open cover of $X$ has a refinement which is a partition.
Indeed, $X$ has a base of cardinality $\omega_1$ which consists of clopen sets.
If $\calU$ is an open cover of $X$, then let $\{V_\alpha\colon\alpha<\omega_1\}$ be a refinement of $\calU$ which consists of clopen sets.
For each $\alpha<\omega_1$, put
$$U_\alpha=V_\alpha\setminus\bigcup\{V_\beta\colon\beta<\alpha\}.$$
The sets $U_\alpha$ constitute the desired  partition.

Under the Continuum Hypothesis, by Lemma \ref{lem:9}, the space $(2^{\omega_1})_\delta$ is of weight $\omega_1=\frak{c}$.
But without the Continuum Hypothesis, any open cover of $(2^{\omega_1})_\delta$ has a refinement which is a partition.
Indeed, the family
$$\calB=\{[f]\colon f\in 2^\alpha\mbox{ and }\alpha<\omega_1\}$$
is a base for $(2^{\omega_1})_\delta$.
If $\calU$ is an open cover of $(2^{\omega_1})_\delta$, then let $\calV\subseteq\calB$ be a refinement of $\calU$.
The family of all maximal elements of $\calV$, with respect to the inclusion, is the desired partition.
In particular, we see that $(2^{\omega_1})_\delta$ is a paracompact space.

If $X$ is a $P$-space, then a family $\frak{M}_X=\{\mathcal{P}_\alpha\colon \alpha<\omega_1\}$ is called $P$\textit{-matrix}, whenever
\begin{enumerate}
\item Each $\mathcal{P}_\alpha$ is a partition of $X$.
\item  If $\beta<\alpha$, then $\mathcal{P}_\alpha$ refines $\mathcal{P}_\beta$.                                    
\item The union $\bigcup\frak{M}_X$ is a base for $X$.
\item If $\alpha$ is an infinite limit ordinal and $U\in\calP_\alpha$, then
  $$U=\bigcap\{V\in\calP_\beta\colon \beta<\alpha\mbox{ and }U\subseteq V\}.$$
\end{enumerate}

If a $P$-space $X$ has a $P$-matrix $\{\calP_\alpha\colon \alpha<\omega_1\}$, then any open cover of $X$ has a refinement  which is a partition, since a slightly modified argument used for $(2^{\omega_1})_\delta$ works.
Indeed, if $\calU$ is an open cover of $X$, then let $\calV\subseteq\bigcup\{\calP_\alpha\colon\alpha<\omega_1\}$ be a refinement of $\calU$.
The family of all maximal elements of $\calV$, with respect to the inclusion, is the desired partition.

\begin{lem}\label{lem:16}
Any $P$-space of weight $\omega_1$ has a $P$-matrix.
\end{lem}

\begin{proof}
Let $X$ be a $P$-space with a base  $\{U_{\alpha+1}\colon \alpha<\omega_1\}$ consisting of clopen sets.
Let $\mathcal{P}_0=\{X\}$.
Assume that partitions $\{\mathcal{P}_\beta\colon \beta<\alpha\}$ are already defined.
Let
$$\mathcal{P}_\alpha^*=\{\bigcap L\colon L\mbox{ is a maximal chain in }\bigcup\{\mathcal{P}_\beta\colon\beta<\alpha\}\}.$$
If $\alpha$ is an infinite limit ordinal, then put $\calP_\alpha=\calP_\alpha^*$.
If $\alpha$ is not a limit ordinal, then let $\calP_\alpha$ be a partition which refines  $\{U_\alpha,X\setminus U_\alpha\}$ and the partition  $\calP_\alpha^*$.
The family $\{\mathcal{P}_\alpha\colon\alpha<\omega_1\}$ is the desired $P$-matrix.
\end{proof}

Let $(\gamma_\alpha)$ be the  increasing enumeration of all countable infinite limit ordinals.
Put $Q_\alpha=2^{\gamma_\alpha}$.
Since $(2^{\omega_1})_\delta=\varprojlim\{2^\alpha,\pi^\alpha_\beta,\omega_1\}$ and the family of countable limit ordinals is cofinal in $\omega_1$, we get $$(2^{\omega_1})_\delta=\varprojlim\{Q_\alpha,\pi^\alpha_\beta,\omega_1\},$$
where ${\pi}^\alpha_\beta\colon Q_\alpha\to Q_\beta$ and $\pi^\alpha_\beta(f)=f|_{\gamma_\beta}$: the symbol $\pi^\alpha_\beta$ has been used in two different meanings, but this does not lead  to confusion.
Note that, each $Q_\alpha$ is of cardinality $\frak{c}$. 

\begin{thm}\label{thm:16}
Any $P$-space  of weight $\omega_1$ can be embedded into $(2^{\omega_1})_\delta$.
\end{thm}

\begin{proof}

Let $X$ be a $P$-space of weight $\omega_1$ and  let $\{\calP_\alpha\colon \alpha<\omega_1\}$ be a $P$-matrix for $X$.
Thus, we have an inverse system $\{\calP_\alpha,r^\alpha_\beta ,\omega_1\}$, where each $r^\alpha_\beta$ is the restriction of the inclusion and each $\calP_\alpha$ is equipped with the discrete topology.
Each $x\in X$ determines the thread in $\varprojlim\{\calP_\alpha,r^\alpha_\beta,\omega_1\}$, since  $\bigcup\{\calP_\alpha\colon\alpha<\omega_1\}$ is a base for $X$.
Hence  $X$   can be embedded into $\varprojlim\{\calP_\alpha,r^\alpha_\beta,\omega_1\}$, by Proposition 2.5.5 \cite{eng}.

To show that  $\varprojlim\{\calP_\alpha,r^\alpha_\beta,\omega_1\}$ can be embedded into $(2^{\omega_1})_\delta=\varprojlim\{Q_\alpha,\pi^\alpha_\beta,\omega_1\}$, where $Q_\alpha=2^{\gamma_\alpha}$, we shall define a sequence of injections  $S_\alpha\colon \calP_\alpha\to Q_\alpha$ such that 
Diagram \ref{fig:3}
      is commutative.
\begin{figure}[h]
$$\xymatrix{
	\calP_\beta \ar@{->}@{-->}[d]_{S_\beta} && \calP_\alpha\ar[ll]_{r^\alpha_\beta} \ar@{->}@{-->}[d]^{S_\alpha} \\
	Q_\beta       && Q_\alpha\ar[ll]_{\pi^\alpha_\beta}
      }$$\caption{}\label{fig:3}
      \end{figure}
      
      Let $S_0\colon \calP_0\to Q_0$ be an arbitrary injection.
Assume that for each $\beta<\alpha$ an injection $S_\beta$ is defined such that appropriate diagrams are commutative, i.e. $S_\gamma\circ r^\beta_\gamma=\pi^\beta_\gamma\circ S_\beta$ for $\gamma<\beta<\alpha$.
If $\alpha$ is a limit ordinal and $U\in \calP_\alpha$, then
$$S_\alpha(U)=\bigcup\{S_\beta(V)\colon V\in\calP_\beta\mbox{ and }V\supseteq U\mbox{ and }\beta<\alpha \}.$$

Let $\alpha=\beta+1$.
For each $V\in \calP_\beta$, let $\calP_V=\{U\in\calP_\alpha\colon U\subseteq V\}$ and let $Q_V=\{f\in Q_\alpha\colon S_\beta(V)\subseteq f\}$.
Choose arbitrary injections $S_V\colon \calP_V\to Q_V$ for each $V\in\calP_\beta$.
Then put  $S_\alpha=\bigcup\{S_V\colon V\in\calP_\beta\}$.

Since each $S_\alpha$ is an injection, hence by Lemma 2.5.9, \cite{eng}, we obtain the desired  embedding.
\end{proof}

Under the Continuum Hypothesis, the space $(2^{\omega_1})_\delta$ has the greatest dimensional type in the class of all $P$-spaces of weight $\omega_1$.
But if the Continuum Hypothesis fails, then our argumentation does not work, since $(2^{\omega_1})_\delta$ is of weight $\frak{c}$.

\section{$P$-spaces of cardinality and weight $\omega_1$}\label{sec:7}

Let $X$ be a  $P$-space with a $P$-matrix $\frak{M}_X=\{\mathcal{P}_\alpha\colon \alpha<\omega_1\}$. 
If $X$ has weight $\omega_1$, then such a $P$-matrix exists by Lemma \ref{lem:16}.
But, if   $Z\subseteq X$ is a  dense subset of cardinality and weight $\omega_1$, then there exists a labeling   $E\colon \bigcup\frak{M}_X\to Z$. 

\begin{lem}\label{lem:18b}
If $X$ is a $P$-space of cardinality and weight $\omega_1$, then  there exists a labeling $E\colon\bigcup\frak{M}_X\to X$ for any $P$-matrix $\frak{M}_X$.
\end{lem}
\begin{proof}
Let  $X=\{x_\alpha\colon \alpha<\omega_1\}$ and let $\frak{M}_X=\{\mathcal{P}_\alpha\colon \alpha<\omega_1\}$ be a $P$-matrix.
  If $V\in\mathcal{P}_\alpha$, then let $\beta=\inf\{\nu\colon x_\nu\in V\}$ and then put  $E(V)=x_\beta$.
The function $E\colon\bigcup\frak{M}_X\to X$ is a labeling, which can be checked as in the proof of Theorem \ref{thm:14}.
\end{proof}
  
 The below theorem, under the Continuum Hypothesis, follows from \cite[Lemma 2.2]{dow}.
 
  \begin{thm}\label{lem:18a}
Let  $Y$ be a $P$-space with a $P$-matrix
$$\frak{M}_Y=\{\mathcal{P}_\alpha\colon\alpha<\omega_1\}$$
such that each $\calP_\alpha$ is of cardinality at most $\frak{c}$.
If there exists a labeling  $E\colon \frak{M}_Y\to Y,$ then $Y$ can be embedded into $\mathbf{\Sigma}$.
\end{thm}

\begin{proof}
Fix a $P$-space $Y$ with a $P$-matrix
$$\frak{M}_Y=\{\mathcal{P}_\alpha\colon\alpha<\omega_1\},$$
such that each $\calP_\alpha$ is of cardinality at most $\frak{c}$.
Let  $E\colon \bigcup\frak{M}_Y\to Y$ be a labeling.
Analogously as in the proof of Theorem \ref{thm:16} we shall  define Diagram \ref{fig:5}, which is a version of Diagram \ref{fig:2}, where $\iota^\alpha_\beta\colon Q_\beta\to Q_\alpha$ and $\pi^\alpha_\beta\colon Q_\alpha\to Q_\beta$ are defined just before Theorem \ref{thm:16}.
\begin{figure}[h]
$$\xymatrix{
	\mathcal{P}_\beta \ar@<1ex>@{--{>}}[rr]^{\eta^\alpha_\beta}\ar@{-->}[d]_{S_\beta} && \mathcal{P}_\alpha\ar[ll]^{r^\alpha_\beta} \ar@{->}@{--{>}}[d]^{S_\alpha} \\
	Q_\beta\ar@<1ex>@{--{>}}[rr]^{\iota^\alpha_\beta}       && Q_\alpha\ar[ll]^{\pi^\alpha_\beta}
      }$$\caption{}\label{fig:5}
      \end{figure}

      But injections $\eta^\alpha_\beta\colon \calP_\beta\to \calP_\alpha$ are determined by the labeling $E\colon\bigcup\frak{M}_Y\to Y$, i.e. $\eta^\alpha_\beta(U)\in\calP_\alpha$ and $r^\alpha_\beta(V)\in\calP_\beta$ are  unique elements such that $E(U)\in \eta^\alpha_\beta(U)$ and $r^\alpha_\beta(V)\supseteq V$, for $U\in\calP_\beta$ and $V\in\calP_\alpha$.
Thus, we have defined two enriched inverse systems $\mathbb{Q}=\{Q_\alpha,\pi^\alpha_\beta,\iota^\alpha_\beta,\omega_1\}$ and $\mathbb{P}=\{\calP_\alpha,r^\alpha_\beta,\eta^\alpha_\beta,\omega_1\}$, so it remains to define injections $S_\alpha$, which will be done by a modification of the proof of Theorem \ref{thm:4}.
      Namely, let $S_0\colon\calP_0\to Q_0$ be an injection.
      Let $\alpha<\omega_1$.
      Assume that we have defined a sequence of injections  $S_\gamma\colon \calP_\gamma\to Q_\gamma$ for $\gamma<\alpha$, such that
       the diagrams obtained from Diagram \ref{fig:5} by replacing $\alpha$ with $\gamma$   are commutative, where $\beta<\gamma<\alpha$.

      If $\alpha=\gamma+1$, then fix $V\in\calP_\gamma$.
      Choose an injection
      $$S^V_\alpha\colon \{U\in\calP_{\alpha}\colon U\subseteq V\}\to \{W\in Q_{\alpha}\colon W\subseteq S_\gamma(V)\}$$
      such that $S^V_\alpha(\eta^\alpha_\gamma(V))=\iota^\alpha_\gamma(S_\gamma(V))$.
Put $S_\alpha(U)=S^V_\alpha(U)$, where $V$ is a unique element of $\calP_\gamma$ containing $U$.

      If $\alpha$ is a limit ordinal and $V\in \calP_\alpha$, then $V=\bigcap\{U\in\calP_\beta\colon \beta<\alpha\mbox{ and }V\subseteq U\},$
      so put
      $$S_\alpha(V)=\bigcup\{S_\beta(U)\colon \beta<\alpha\mbox{ and }V\subseteq U\in\calP_\beta\}.$$
      By Lemma 2.5.9 \cite{eng}, the function $S\colon\varprojlim\mathbb{P}\to\varprojlim\mathbb{Q}$, given by the formula $S((x_\alpha))=(S_\alpha(x_\alpha))$, is an embedding such that $$S|_{\mathbf{\Sigma}_{\mathbb{P}}}\colon\mathbf{\Sigma}_{\mathbb{P}}\to \mathbf{\Sigma}_{\mathbb{Q}}=\mathbf{\Sigma}.$$
      The proof is completed, since  $Y$ has to be homeomorphic to $\mathbf{\Sigma}_{\mathbb{P}}$.
      Indeed, if  $f_\alpha\colon Y\to \calP_\alpha$ are functions such that $x\in f_\alpha(x)$, then the function $f\colon Y\to \mathbf{\Sigma}_\mathbb{P}$, given by the formula $x\mapsto f(x)=(f_\alpha(x))$, is a homeomorphism.
      \end{proof}

\begin{cor}\label{thm:17}
Any $P$-space of cardinality and weight $\omega_1$ can be embedded into the space $\mathbf{\Sigma}$.
\end{cor}

\begin{proof}
By Lemmas \ref{lem:16} and \ref{lem:18b}, any $P$-space of cardinality and weight $\omega_1$ has a $P$-matrix and a labeling as it is required in Theorem \ref{lem:18a}.
\end{proof}

If $X$ and $Y$ are topological spaces, then
  $X\to (Y)^1_2$ means that $Y$ can be embedded into one of  $A$ or $B$  for any subspaces $A$ and $B$ such that   $X=A\cup B$.
If $Z\subseteq (2^{\omega_1})_\delta$ is a dense subset, then any $P$-space of cardinality and weight $\omega_1$ can be embedded into $Z$, see A. Dow  \cite{dow}. 
Thus, Theorem \ref{thm:4} and Corollary \ref{thm:17} provide another proof of Dow's result.
Also, Theorem \ref{thm:4} implies $(2^{\omega_1})_\delta\to (\mathbf{\Sigma})^1_2$, which gives an example concerning Question 6.2 stated in \cite{com}.

\section{On Lindel\"of and nowhere Lindel\"of $P$-spaces}\label{sec:8}
Recall  that a  space is \textit{Lindel\"of} if  its  every open cover has a countable subcover.
We say that a topological space is  \textit{nowhere Lindel\"of}, whenever it does not contain a non-empty open subset with the Lindel\"of property.
Assume  that $\lambda$ is an infinite cardinal number and a $P$-matrix $\{\calP_\alpha\colon\alpha<\omega_1\}$ satisfies conditions (1)--(4).
We shall add another condition.

 (5-$\lambda$). Each $\calP_\alpha$ is of cardinality $\lambda$ and if $\beta<\alpha$, then
 any $V\in\mathcal{P}_\beta$ contains $\lambda$ many elements of $\mathcal{P}_\alpha$.
 
\noindent We are particularly interested in $\lambda\in\{\omega,\omega_1,\frak{c}\}$.
If $X$ is a $P$-space with a $P$-matrix $\{\calP_\alpha\colon\alpha<\omega_1\}$ which satisfies condition (5-$\lambda$), then we have an inverse system $\{\calP_\alpha,r^\alpha_\beta,\omega_1\}$ defined analogously as $\{Q_\alpha,\pi^\alpha_\beta,\omega_1\}$ just before Theorem \ref{thm:16}. 
If each $\calP_\alpha$ is equipped with the discrete topology and
$r^\alpha_\beta\colon\calP_\alpha\to\calP_\beta$, where $r^\alpha_\beta(U)\in\calP_\beta$ is a unique element containing $U$, then we get the inverse limit $$\varprojlim\{\calP_\alpha,r^\alpha_\beta,\omega_1\},$$
which contains a homeomorphic copy of $X$ as a dense subset.
Thus, $\varprojlim\{\calP_\alpha,r^\alpha_\beta,\omega_1\}$ is  a crowded $P$-space of weight and density equal to $$\max\{\lambda, \omega_1\}=|\bigcup\{\calP_\alpha\colon\alpha<\omega_1\}|.$$

\begin{pro}
If a   Lindel\"of (nowhere Lindel\"of) $P$-space $X$ is of cardinality and weight $\omega_1$, then $X$ has a $P$-matrix $\{\mathcal{P}_\beta\colon \beta<\omega_1\}$ such that if $\beta<\alpha$, then  any $V\in\mathcal{P}_\beta$ contains countably (respectively $\omega_1$)  many elements of $\mathcal{P}_\alpha$.
\end{pro}

\begin{proof}
If $X$ is a Lindel\"of space, then a $P$-matrix constructed as in the proof of Lemma \ref{lem:16} is suitable.
But, if  $X$ is a nowhere Lindel\"of space, it suffices to modify the construction of a $P$-matrix $\{\calP_\alpha\colon\alpha<\omega_1\}$ from the proof of Lemma \ref{lem:16}, defining a new $P$-matrix $\{\calP^*_\alpha\colon\alpha<\omega_1\}$.
Namely, each partition $\calP^*_{\alpha+1}$ is  such that any $V\in\calP^*_\alpha$ contains  $\omega_1$ many elements of $\calP^*_{\alpha+1}$.
But, if $\alpha$ is a limit ordinal, then $\calP^*_\alpha$ is defined analogously as $\calP_\alpha$.
\end{proof}

Let $X$ be a $P$-space with a $P$-matrix  $\frak{M}_X=\{\mathcal{P}_\alpha\colon\alpha<\omega_1\}$ which satisfies condition (5-$\lambda$).
Suppose that there exists a labeling $E\colon \bigcup\frak{M}_X\to X$ which satisfy the following condition.

$(*)$. \textit{ If $\alpha<\omega_1$ and $\mathcal{L}$ is a chain contained in $\bigcup\{\mathcal{P}_\beta\colon \beta<\alpha\}$ and $\bigcap\mathcal{L}\neq\emptyset$, then there exists $V\in\mathcal{L}$ such that $E(V)\in\bigcap\mathcal{L}$.}

\noindent In this case we say that $X$ has a $\lambda$\textit{-thin labeling}.
By the definition, any $P$-space with a $\lambda$-thin labeling is  a crowded space since every base set contains infinitely  many pairwise disjoint subsets.
Also, if $X$ has a $\lambda$-thin labeling $E\colon\frak{M}_X\to X$, then
$$E[\calP_\alpha]=E[\bigcup\{\calP_\beta\colon \beta<\alpha\}],$$
 for each limit ordinal $\alpha$.
Applying  Lemma \ref{lem:I5}, one can check that  an arbitrary $P$-space with a $\lambda$-thin labeling has to be of first category.

\begin{lem}\label{lem:19}
  If a crowded $P$-space $X$ is of weight $\omega_1$, then there exists $Z\subseteq X$ such that $Z$ has an $\omega$-thin labeling.
\end{lem}

\begin{proof}
Applying Lemma \ref{lem:16},  choose a dense subset $Z\subseteq X$ of  cardinality $\omega_1$ with a $P$-matrix $\frak{M}_Z=\{\mathcal{Q}_\alpha\colon \alpha<\omega_1\}$.
Let $E\colon\bigcup\frak{M}_Z\to Z$ be a labeling, which exists by Lemma \ref{lem:18b}.
 Without loss of generality, because $Z$ is crowded, assume that if $\beta<\alpha$ and $V\in\mathcal{Q}_\beta$,  then $V$ contains infinitely many elements of $\mathcal{Q}_\alpha$.
  Choose a family  $\mathcal{P}_0\subseteq \mathcal{Q}_0$ such that  $|\mathcal{P}_0|= \omega$.
  For each $V\in\mathcal{P}_0$,  choose a point $E(V)\in V$.
  Suppose families $\{\mathcal{P}_\beta\colon \beta<\alpha\}$ and points
  $$\{E(V)\colon V\in\bigcup\{\mathcal{P}_\beta\colon\beta<\alpha\}\}$$
  are defined.
  If $V\in\bigcup\{\mathcal{P}_\beta\colon\beta<\alpha\}$, then put $$\mathcal{L}_V=\{W\in\bigcup\{\mathcal{P}_\beta\colon \beta<\alpha\}\colon E(V)\in W\}.$$
  Then choose a family $\mathcal{R}_V\subseteq\mathcal{Q}_{\alpha+1}$ consisting of $\omega$ many pairwise disjoint clopen subsets of $\bigcap\mathcal{L}_V\subseteq V$ such that $E(V)\in\bigcup\mathcal{R}_V$.
  For each $W\in\mathcal{R}_V$ such that $E(V)\notin W$,  choose a point $E(W)\in W$.
  Let
  $$\mathcal{P}_{\alpha+1}=\bigcup\{\mathcal{R}_V\colon \mathcal{R}_V\subseteq\mathcal{Q}_{\alpha+1}\mbox{ and }V\in\bigcup \{\mathcal{P}_\beta\colon\beta<\alpha\}\}.$$
  Let $Z$ be the set of all points $E(V)$, which are defined above.
  Any base set of $Z$ contains infinitely many pairwise disjoint subsets, hence  $Z$ is crowded.
  Putting $\mathcal{P}^*_\alpha=\{V\cap Z\colon V\in\mathcal{P}_\alpha\}$, define the function $E^*\colon \bigcup\{\mathcal{P}^*_\alpha\colon\alpha<\omega_1\}\to Z$ by the formula $V\cap Z\mapsto E(V)$.
  The map $E^*$  is a labeling.
  \end{proof}

  \begin{thm}\label{thm:25}
  If a $P$-space $Y$  has an $\omega$-thin labeling, then $Y$ is a Lindel\"of space.
\end{thm}
\begin{proof}
  Let $\{\calP_\alpha\colon\alpha<\omega_1\}$ be a $P$-matrix for $Y$ and
  $$E\colon \bigcup\{\calP_\alpha\colon\alpha<\omega_1\}\to Y$$
  be an $\omega$-thin labeling.
Fix an open cover $\mathcal{U}$ of $Y$.
We can assume that $\mathcal{U}$ is a partition of $Y$,  since $Y$ is of cardinality and weight $\omega_1$, see Lemma \ref{thm:6}.
Let $\alpha_0<\omega_1$ be an ordinal number such that if $V\in\mathcal{P}_{0}$, then there exist $\beta\leq\alpha_0$, $W_V\in\mathcal{P}_{\beta}$ and $U_V$ such that
$$E(V)\in W_V\subseteq U_V\in\mathcal{U}.$$
  Assume that an ordinal $\alpha_{n}$ is defined such that if $V\in\bigcup\{\calP_{\beta}\colon \beta\leq\alpha_{n-1}\}$, then there exist $\gamma\leq \alpha_{n}$, $W_V\in\mathcal{P}_{\gamma}$ and $U_V$ such that
$$E(W_V)=E(V)\in W_V\subseteq U_V\in\mathcal{U}.$$
Let $\alpha_{n+1}>\alpha_{n}$ be a countable ordinal such that if $V\in\bigcup\{\mathcal{P}_\beta\colon \beta\leq\alpha_{n}\}$, then there $W_V\in \bigcup\{\mathcal{P}_{\beta}\colon \beta\leq \alpha_{n+1}\}$ and $U_V$, fulfilling
$$E(V)\in W_V\subseteq U_V\in\mathcal{U}.$$
Since $E$ is an $\omega$-thin labeling, using conditions (5-$\omega$) and ($*$), we check that if $\alpha=\sup\{\alpha_n\colon n>0\}$, then the partition $\calP_\alpha$  refines $\mathcal{U}$.
\end{proof}

Translating the above proof to the language of category theory, we get that $P$-matrix for $Y$ is a Fra\"{\i}ss\'e   $\omega_1$-sequence in the category of all open covers of $Y$, with refining pairs of covers as morphisms.
We recommend the paper \cite{kub} for details about Fra\"{\i}ss\'e sequences.

If $\mathbb{P}=\{\calP_\alpha,\pi^\beta_\alpha,\omega_1\}$ is an inverse system, where $\calP_\alpha$ are countable discrete spaces, then $\varprojlim\mathbb{P}$ is not necessary a Lindel\"of space.
This has been observed in \cite{jw}, compare \cite[Lemma 2]{kun}.
Let us present a sketch of proof.
Let $\mathbb{T}_A$ be an Aronszajn tree.
Let $\{\ell_\alpha\colon\alpha<\omega_1\}$ be a sequence of branches of $\mathbb{T}_A$ with different height.
Then, each $\ell_\alpha$ is extended by a copy of a tree determined by a $P$-space which has an $\omega$-thin  labeling.
The family of all just extended branches gives a tree  of height $\omega_1$ with all levels countable, i.e. we get the desired  inverse limit which is not Lindel\"of.

\begin{pro}\label{pro:19}
  If a crowded $P$-space $X$ is of weight $\omega_1$, then there exists $Y\subseteq X$ such that $Y$ is nowhere Lindel\"of.
\end{pro}

\begin{proof}
  Let $X$ be a crowded $P$-space with a $P$-matrix  $\{\mathcal{Q}_\alpha\colon \alpha<\omega_1\}$.
  If $V\in\mathcal{Q}_0$, then choose  a family  $\mathcal{P}_V$  consisting of $\omega_1$ pairwise disjoint open sets such  that $\bigcup\mathcal{P}_V\subseteq V$.
  Let $\mathcal{P}_0=\bigcup\{\mathcal{P}_V\colon V\in\mathcal{Q}_0\}$ and let $Y_0\subseteq X$ be such that $Y_0\cap V$ is a singleton for each $V\in\mathcal{P}_0$.
  Assume that families $\{\mathcal{P}_\beta\colon\beta<\alpha\}$ are defined.
  If $V\in\mathcal{Q}_\alpha$ and there exists $W_\beta\in\mathcal{P}_\beta$, for each $ \beta<\alpha$, such that
  $$V\cap\bigcap\{W_\beta\colon\beta<\alpha\}\neq\emptyset,$$
  then choose  a family  $\mathcal{P}_V$ consisting of $\omega_1$ pairwise disjoint open sets such  that $\bigcup\mathcal{P}_V\subseteq V\cap\bigcap\{W_\beta\colon\beta<\alpha\}$, otherwise $\mathcal{P}_V=\emptyset$.

  Let
  $$\mathcal{P}_\alpha=\bigcup\{\mathcal{P}_V\colon V\in\mathcal{Q}_\alpha\}$$ and let $Y_\alpha\subseteq X$  be such that $\bigcup\{Y_\beta\colon\beta<\alpha\}\subseteq Y_\alpha$ and  $Y_\alpha\cap V$ is a singleton for each $V\in\mathcal{P}_\alpha$.
  Thus $\{\mathcal{P}_\alpha\colon\alpha<\omega_1\}$ is a $P$-matrix for the subset $Y=\bigcup\{Y_\alpha\colon\alpha<\omega_1\}\subseteq X$.
  Put $E(V)=Y_\alpha\cap V$, whenever $V\in\mathcal{P}_\alpha$.
  Thus,  $E\colon \bigcup\{\mathcal{P}_\alpha\colon \alpha<\omega_1\}\to Y$ is a desired labeling.
\end{proof}

\begin{cor}\label{cor:19}
  If a crowded $P$-space $X$ is of weight $\omega_1$, then there exists $Z\subseteq X$ such that $Z$ has an $\omega_1$-thin labeling.
  \end{cor}

  \begin{proof}
Using Proposition \ref{pro:19}, take $Y\subseteq X$ such that $Y$ is a nowhere Lindel\"of subspace.  
  Let  $\{\mathcal{Q}_\alpha\colon \alpha<\omega_1\}$ be a $P$-matrix for $Y$ such that if $\alpha<\beta$ and $V\in\mathcal{Q}_\alpha$,  then $V$ contains $\omega_1$ many elements of $\mathcal{Q}_\beta$.
  The rest of the proof is a modification of the reasoning of the proof of Lemma \ref{lem:19}.
Namely, the family $\mathcal{P}_0\subseteq \mathcal{Q}_0$ is chosen to be of cardinality $\omega_1$.
  If $V\in\mathcal{P}_0$, then select a point $E(V)\in V$.
Assume that  families $\{\mathcal{P}_\beta\colon \beta<\alpha\}$ and points
  $$\{E(V)\colon V\in\bigcup\{\mathcal{P}_\beta\colon\beta<\alpha\}\}$$
  are defined.
  If $V\in\bigcup\{\mathcal{P}_\beta\colon\beta<\alpha\}$, then we repeat the definition of  $$\mathcal{L}_V=\{W\in\bigcup\{\mathcal{P}_\beta\colon \beta<\alpha\}\colon E(V)\in W\}.$$
  Then a family  $\mathcal{R}_V\subseteq\mathcal{Q}_{\alpha+1}$ is chosen such that it consists of $\omega_1$ many pairwise disjoint clopen subsets of $\bigcap\mathcal{L}_V\subseteq V$ and  $E(V)\in\bigcup\mathcal{R}_V$.
  Let
  $$\mathcal{P}_\alpha=\bigcup\{\mathcal{R}_V\colon \mathcal{R}_V\subseteq\mathcal{Q}_{\alpha+1}\mbox{ and }V\in\bigcup \{\mathcal{P}_\beta\colon\beta<\alpha\}\}.$$
  If $W\in\mathcal{R}_V$ and $E(V)\notin W$, then choose a point $E(W)\in W$.
  Let $Z$ be the set of all points $E(V)$, which are defined above.
  By the definition, any base set of  $Z\subseteq X$ has a partition consisting of  $\omega_1$  clopen subsets, then $Z$ is nowhere Lindel\"of.

  Since $E\colon \bigcup\{\mathcal{P}_\alpha\colon\alpha<\omega_1\}\to Z$ is a surjection, then $Z$ has an $\omega_1$-thin labeling.
\end{proof}

Now, we can prove counterparts of Theorem \ref{thm:14}.

\begin{thm}\label{thm:23}
If $\lambda$ is an infinite cardinal number, then 
any two $P$-spaces which have $\lambda$-thin labelings  are homeomorphic.
\end{thm}
\begin{proof}
Assume that $X$ and $Y$ have $\lambda$-thin labelings.
Let $\frak{M}_X=\{{Q}_\alpha\colon\alpha<\omega_1\}$ be a $P$-matrix of $X$ with a $\lambda$-thin labeling $E\colon\bigcup\frak{M}_X\to X$ and let $\frak{M}_Y=\{{R}_\alpha\colon\alpha<\omega_1\}$ be a $P$-matrix of $Y$ with a $\lambda$-thin labeling $F\colon\bigcup\frak{M}_Y\to Y$.
Thus, we have two enriched systems $\mathbb{Q}=\{Q_\alpha,q^\alpha_\beta,\iota^\alpha_\beta,\omega_1\}$ and $\mathbb{R}=\{R_\alpha,r^\alpha_\beta,\eta^\alpha_\beta,\omega_1\}$, where $\iota^\alpha_\beta$ and $\eta^\alpha_\beta$ are determined by $\lambda$-thin labelings $E$ and $F$, respectively, but $q^\alpha_\beta$ and $r^\alpha_\beta$ are determined by the inclusion. 
We shall define a bijection $s_\alpha\colon Q_\alpha\to R_\alpha$ such that the following diagram
$$\xymatrix{
	Q_\beta \ar@<1ex>@{--{>}}[rr]^{\iota^\alpha_\beta}\ar@{-->}[d]_{s_\beta} && Q_\alpha\ar[ll]^{q^\alpha_\beta} \ar@{->}@{--{>}}[d]^{s_\alpha} \\
	R_\beta\ar@<1ex>@{--{>}}[rr]^{\eta^\alpha_\beta}       && R_\alpha\ar[ll]^{r^\alpha_\beta}
      }$$
is commutative, whenever $\beta<\alpha<\omega_1$.

If $\alpha=\gamma+1$, then fix $V\in Q_\gamma$.
      Choose an injection
      $$s^V_\alpha\colon \{U\in Q_{\alpha}\colon U\subseteq V\}\to \{W\in R_{\alpha}\colon W\subseteq s_\gamma(V)\}$$
      such that $s^V_\alpha(\iota^\alpha_\gamma(V))=\eta^\alpha_\gamma(s_\gamma(V))$.
Put $s_\alpha(U)=s^V_\alpha(U)$, where $V$ is a unique element of $Q_\gamma$ containing $U$.

      If $\alpha$ is a limit ordinal and $V\in Q_\alpha$, then, by condition $(*)$, it follows that there exists $\gamma<\alpha$ and $U\in Q_\gamma$ such that $E(V)=E(U)$, hence we define $s_\alpha(V)=s_\gamma(U)$.

      By Lemma \ref{lem:I}, the inverse limits $\varprojlim\mathbb{Q}$ and $\varprojlim\mathbb{R}$ are homeomorphic.
      Condition $(*)$ and Lemma \ref{lem:I5} imply that $X=\varprojlim\mathbb{Q}$ and $Y=\varprojlim\mathbb{R}$.
\end{proof}

Because of condition $(*)$, if a $P$-space $X$ has a $\lambda$-thin labeling, then $X$ is  an inverse limit,  as it focused on at the end of the above proof.

\begin{cor}\label{cor:25}
If a $P$-space $Y$ has a $\lambda$-thin labeling and a subset $Z\subseteq Y$ is non-empty and clopen, then  $Z\subseteq Y$ also has  a $\lambda$-thin labeling.
\end{cor}

\begin{proof}
Assume that  $\{P_\alpha\colon \alpha<\omega_1\}$ is a $P$-matrix and $E\colon \bigcup\{P_\alpha\colon\alpha<\omega_1\}\to Y$ is a $\lambda$-thin labeling.
Consider the family
$$R=\{V\subseteq Z\colon V\in P_\alpha\mbox{ and }0<\alpha<\omega_1\}.$$
Let $R_0\subseteq R$ be a maximal family of cardinality $\lambda$,  consisting of pairwise disjoint sets.
Assume that families $\{R_\beta\subseteq R\colon \beta<\alpha\}$ are defined.
Fix a maximal chain $\mathcal{L}\subseteq\bigcup\{R_\beta\colon \beta<\alpha\}$.
Let $L\subseteq \{V\subseteq \bigcap\mathcal{L}\colon V\in R\}$ be the set of all maximal subsets with respect to inclusion and let $R_{\alpha}$ be the union of all just defined families $L$.
Check that $E|_{\bigcup\{R_\alpha\colon\alpha<\omega_1\}}$ is a $\lambda$-thin labeling for $Z$.
\end{proof}

Thus we have the following facts about dimensional types of crowded $P$-spaces of cardinality and weight $\omega_1$.

\begin{thm}\label{thm:26}
Assume that  a $P$-space $X$ of cardinality and weight $\omega_1$ is crowded.
  If a $P$-space $Y$ has an $\omega$-thin labeling and a $P$-space $Z$ has an  $\omega_1$-thin labeling, then  $Y=_h Z\subseteq_h X.$
\end{thm}

\begin{proof}
  By Lemma \ref{lem:19} and Corollary \ref{cor:19}, any crowded $P$-space $X$ of cardinality and weight $\omega_1$ contains copies of a space with  an $\omega$-thin labeling and a space with an $\omega_1$-thin labeling.
Theorem \ref{thm:23} implies that $Y=_h Z\subseteq_h X.$
\end{proof}

\begin{cor}
If a $P$-space $X$ has an $\omega$-thin labeling (or an $\omega_1$-thin labeling), then $X$ has the smallest dimensional type in the class of all crowded $P$-spaces of weight $\omega_1$.
\end{cor}

\begin{proof}
Any crowded $P$-space $X$ of weight $\omega_1$ contains a dense subset  $Z\subseteq X$ of cardinality $\omega_1$.
Since  $Z$ is a $T_1$-space, it is  crowded and hence $Z$ contains subspaces which have an $\omega$-thin labeling and an $\omega_1$-thin labeling.
\end{proof}

\section{Remarks on rigid $P$-spaces}\label{sec:9}


K. Kunen showed that there exists a rigid Lindel\"of $P$-space of cardinality and weight $\omega_1$, see  \cite[2.1. Theorem]{kun}. Let us add a few remarks about rigid Lindel\"of $P$-spaces.

\begin{pro}
There exist at least $\frak c$ many (rigid) $P$-spaces of cardinality and weight $\omega_1$ such that any two of them are not homeomorphic.
\end{pro}

\begin{proof}
Let $X$ be a rigid Lindel\"of $P$-space.
Choose an infinite family $\{U_n\colon n<\omega\}$ of pairwise disjoint clopen subsets of $X$.
Assign each $A\subseteq\omega$ a subspace  $X_A=\bigcup\{U_n\colon n\in A\}\subseteq X$, which is a clopen subset.
If $A\neq B$, then $X_A$ and $X_B$ cannot be homeomorphic.
Indeed, if $A\setminus B\neq\emptyset$ and $h\colon X_A\to X_B$ is a homeomorphism, then $H\colon X\to X$, given by the formula
    $$H(x)=
    \begin{cases}
      h(x),&\mbox{if }x\in X_A\setminus X_B;\\
          h^{-1}(x),&\mbox{if }x\in h[X_A\setminus X_B];\\
    x,&\mbox{otherwise},\\
  \end{cases}$$
 is a non-trivial homeomorphism.
 So, $X$ is not a rigid space.
 
 Thus, if $X$ is a rigid $P$-space constructed by Kunen \cite{kun}, then spaces $\{X_A\colon A\subseteq\omega\}$  are  of cardinality and weight $\frak{c}$, whenever $A\neq\emptyset$, and also are rigid, Lindel\"of and not homeomorphic.
\end{proof}

\begin{cor}
If a rigid $P$-space $X$ is of cardinality and weight $\omega_1$, then a closed subset of $X$, which has an $\omega$-thin labeling or an $\omega_1$-thin labeling, is a nowhere dense set.
\end{cor}

\begin{proof}
Suppose a closed subset $Y\subseteq X$ is  not nowhere dense.
Choose two disjoint subsets $U,V\subseteq Y$, which are  clopen in $X$.
If $Y$ has an $\omega$-thin labeling ($\omega_1$-thin labeling), then $U$ is homeomorphic to $V$, since Theorem \ref{thm:23} and Corollary \ref{cor:25}, which contradicts the rigidity of $X$.
\end{proof}

\begin{thm}
If a $P$-space $X$ of cardinality and weight $\omega_1$ is rigid and a $P$-space $Y$ has an $\omega$-thin labeling, then the relation $X\subset_h Y$ is not fulfilled.
\end{thm}

\begin{proof}
It suffices to show that, if there is an embedding of $X$ into a $P$-space $Y$, which has an $\omega$-thin labeling, then $X$ contains a clopen subset, which has an $\omega$-thin labeling.
Indeed, if $f\colon X\to Y$ is an embedding, then the image $f[X]$ has a $P$-matrix, which is defined as follows.
Assume that $X=\{x_\alpha\colon\alpha<\omega_1\}$ and let $\{\calP_\alpha\colon\alpha<\omega_1\}$ be a  $P$-matrix for $Y$, and let $E\colon \bigcup\{\calP_\alpha\colon\alpha<\omega_1\}\to Y$ be an $\omega$-thin labeling.
Choose a maximal and pairwise disjoint family $\calR_0\subseteq\bigcup\{\calP_\alpha\colon\alpha<\omega_1\}$ such that $f[X]\subseteq \bigcup\calR_0$ and  if $V\in\calR_0$, then $E(V)\in f[X]$ and also there exists $V\in\calR_0$ such that $E(V)=f(x_0)$.
Suppose that families $\{\calR_\beta\colon\beta<\alpha\}$ are defined.
If $V\in\bigcup\{\calR_\beta\colon\beta<\alpha\}$, then let
$$L_V=\bigcap\{W\in\bigcup\{\calR_\beta\colon\beta<\alpha\}\colon E(W)=E(V)\}$$
and then choose a family $R_V\subseteq\bigcup\{\calP_\beta\colon\beta<\omega_1\}$ such that
\begin{itemize}
\item $R_V$ consists of pairwise disjoint sets;
\item $R_V$ is of the maximal possible cardinality, i.e. $\omega_1$ or  $\omega$;
  \item $f[X]\cap L_V\subseteq\bigcup R_V\subseteq L_V$;
\item If $f(x_\alpha)\in V$, then there exists $W\in\calR_V$ such that $E(W)=f(x_\alpha)$;
  \item If $W\in R_V$, then $E(W)\in f[X]$.
    \end{itemize}
    Let $\calR_\alpha$ be the union of above defined families $R_V$, here we assume that if $L_V=L_W$, then $R_V=R_W$.
    For any $\alpha<\omega_1$, put
    $$\calQ_\alpha=\{V\cap f[X]\colon V\in\calR_\alpha\}.$$
    Thus $\{\calQ_\alpha\colon\alpha<\omega_1\}$ constitute a $P$-matrix for $f[X]\subseteq Y$  such that if $E^*(V\cap f[X])=E(V)$, then
    $$E^*\colon\bigcup\{\calQ_\alpha\colon\alpha<\omega_1\}\colon \to f[X]$$
    is a labeling.
    If all families $R_V$ are of cardinality $\omega_1$, then $f[X]$ has an $\omega_1$-thin labeling; a contradiction.
If there exists $R_V$ of cardinality $\omega$, then $f[X]$ contains the clopen subset (in the topology inherited on $f[X]$), which has an $\omega$-thin labeling; again a contradiction.
\end{proof}


\begin{thebibliography}{40}
\bibitem{bps} B. Balcar, J. Pelant, and P. Simon, \textit{The space of ultrafilters on N covered by nowhere dense sets}.  Fund. Math.  110  (1980),  no. 1, 11--24.

\bibitem{com} W. W. Comfort, A. Kato and S. Shelah, \textit{
Topological partition relations of the form $\omega^*\to(Y)^1_2$}, Papers on general topology and applications (Madison, WI, 1991),  Ann. New York Acad. Sci., 704, New York Acad. Sci., New York, 1993, 70--79.

\bibitem{dow} A. Dow, \textit{More topological partition relations on $\beta\omega$}. Topology Appl. 259 (2019), 50--66.

\bibitem{eng} R. Engelking, \textit{General topology}. Mathematical Monographs, Vol. 60, PWN--Polish Scientific Publishers, Warsaw, (1977).



\bibitem{fre} M. Fr\'echet, \textit{Les dimensions d'un ensemble abstrait}.  Math. Ann.  68  (1910),  no. 2, 145--168.

\bibitem{gil} L. Gillman and M. Henriksen, \textit{Concerning rings of continuous functions}, Trans. Amer. Math. Soc. 77 (1954), 340--362.

\bibitem{jec} T. Jech, \textit{Set theory}, Springer Monographs in Mathematics, Springer-Verlag, Berlin (2003), The third millennium edition, revised and expanded.

\bibitem{jw} I. Juh\'asz and W. Weiss, \textit{On a problem of Sikorski}, Fund. Math. 100 (1978) 223--227.

\bibitem{kub} W. Kubi\'s, \textit{Fra\"{\i}ss\'e sequences: category-theoretic approach to universal homogeneous structures}. Annals of Pure and Applied Logic 165 (2014) 1755--1811.

\bibitem{kul} W. Kulpa and A. Szymański, \textit{Decompositions into nowhere dense sets}.  Bull. Acad. Polon. Sci. S\'er. Sci. Math. Astronom. Phys.  25  (1977),  no. 1, 37--39.

\bibitem{kun} K. Kunen, \textit{Rigid $P$-spaces}. Fund. Math. 133(1) (1989), 59--65.

\bibitem{kur} K. Kuratowski, \textit{Topology}. Vol. I. Academic Press, New York-London; Pa\'nstwowe Wydawnictwo Naukowe, Warsaw (1966).

\bibitem{lel} A. Lelek, \textit{On totally paracompact metric spaces}, Proc. Amer. Math. Soc. 19 (1968), 168--170.

\bibitem{mis} A. K. Misra, \textit{A topological view of P-spaces}. General Topology and Appl. 2 (1972), 349--362.

\bibitem{pw} Sz. Plewik and M. Walczy\'nska, \textit{On metric $\sigma$-discrete spaces}. Algebra, logic and number theory, 239--253, Banach Center Publ., 108, Polish Acad. Sci. Inst. Math., Warsaw, (2016).

\bibitem{rud}  M. E. Rudin, \textit{Lectures on set theoretic topology}, Regional Conf. Ser. in Math. no. 23, Am. Math. Soc., Providence, RI, (1975).

\bibitem{sie} W. Sierpi\'nski, \textit{General topology}. Translated by C. Cecilia Krieger. Mathematical Expositions, No. 7, University of Toronto Press, Toronto, 1952.

\end{thebibliography}
\end{document}